\documentclass[12pt,reqno]{amsart}
\usepackage{amsmath,amsfonts,epsfig,amssymb,graphics}

\usepackage[usenames]{color}

\newtheorem{thm}{Theorem}[section]
\newtheorem{cor}{Corollary}[section]
\newtheorem{prop}{Proposition}[section]
\newtheorem{lemma}{Lemma}[section]
\newtheorem{exa}{Example}[section]

\newcommand{\xxx}[1]{#1}
\newcommand{\R}{\mathbb R}
\newcommand{\bbS}{\mathbb S}
 
\numberwithin{equation}{section}
\numberwithin{table}{section} 
\numberwithin{figure}{section}

\begin{document}

\title[Spectrum of hypersurfaces and isoperimetric ratio]{Isoperimetric control of the
  spectrum of a compact hypersurface} 
\author{Bruno Colbois}
\address{Universit\'e de Neuch\^atel, Institut de Math\'ematiques, Rue
  Emile-Argand 11, Case postale 158, 2009 Neuch\^atel Switzerland}
\email{bruno.colbois@unine.ch}
\author{Ahmad El Soufi}
\address{Laboratoire de Math\'{e}matiques et Physique Th\'{e}orique,
  UMR-CNRS 6083, Universit\'{e} François Rabelais de Tours, Parc de
  Grandmont, 37200 Tours, France} 
\email{elsoufi@univ-tours.fr}
\author{Alexandre Girouard}
\address{Universit\'e de Neuch\^atel, Institut de Math\'ematiques, Rue
  Emile-Argand 11, Case postale 158, 2009 Neuch\^atel
  Switzerland}
\email{alexandre.girouard@unine.ch}

\thanks{The second author has benefited from the support of the ANR
  (Agence Nationale de la Recherche) through FOG project
  ANR-07-BLAN-0251-01.}

\date{\today}
\begin{abstract}
  Upper bounds for the eigenvalues of the Laplace-Beltrami operator on a
  hypersurface bounding a domain in some ambient Riemannian manifold are
  given in terms of the isoperimetric ratio of the domain.  These
  results are applied to the extrinsic geometry of isometric
  embeddings.
\end{abstract}




\maketitle

\section{Introduction} 
The spectrum of the Laplace-Beltrami operator on a compact Riemannian
manifold $(\Sigma,g)$ of dimension $n\ge 2$ provides a sequence of
global Riemannian invariants
$$
0 = \lambda_1(\Sigma)\le\lambda_2(\Sigma) \leq
\lambda_3(\Sigma) \leq \cdots  \nearrow\infty.
$$
One of the main goals of  spectral geometry is to investigate
relationships between  these invariants  and  other geometric data of
the manifold $\Sigma$ such as the volume, the diameter, the curvature,
or the Cheeger isoperimetric constant. See~\cite{berard,bgm,cha1,
  dav} for classical references.


\medskip
Since the work of Bleecker and Weiner, Reilly and others,  the
following approach has been developed : the manifold $(\Sigma,g)$ is
immersed isometrically into Euclidean space, or a more general
ambient space. One then looks for relationships between the
eigenvalues $\lambda_k(\Sigma)$ and extrinsic geometric quantities
constructed from the second fundamental form of the immersed
submanifold, such as the length of the mean curvature vectorfield. See
for example \cite{BW, ehei,  EI, grosjean,hein,reilly}. It is worth noticing that the spectrum of $(\Sigma, g)$ cannot be
controlled only by the volume of $(\Sigma, g)$
(see~\cite {CD, colboisElSoufi, lohkamp}), even for isometrically
embedded hypersurfaces (see \cite[Theorem 1.4]{cde}).

\medskip
More recently, the first two authors and E. Dryden \cite{cde} have
obtained upper estimates for all  normalized eigenvalues
$\lambda_k(\Sigma)\vert \Sigma\vert^{2/n},$
where $\vert \Sigma\vert$ denotes the Riemannian volume of
$\Sigma$,  in terms of the number of intersection points of  the
immersed submanifold  with a generic affine plane of complementary
dimension. Such results allow a better understanding of the geometry
of a Riemannian metric $g$ on  $\Sigma$ inducing large eigenvalues,
that is such that  for some $k\ge 2$, the $k$-th normalized eigenvalue
$\lambda_k(\Sigma,g) {\vert (\Sigma,g) \vert}^{2/n}$ is large.
Indeed, if $g$ is such a metric, then  any  isometric immersion of
$(\Sigma,g)$  into the Euclidean space $\R^{n+p}$ must have a large
mean curvature, at least somewhere, and a large number of intersection
points with some $p$-planes.

\medskip
In the same vein, Reilly \cite[Corollary 1]{reilly} and
Chavel~\cite{cha2}  obtained the following remarkable 
inequality for the first positive eigenvalue $\lambda_2(\Sigma)$ in
the case where $\Sigma$ is embedded as a hypersurface bounding a
domain $\Omega$ in $\R^{n+1}$ (or in a Cartan-Hadamard manifold
in~\cite{cha2}):
\begin{equation}\label{chavel}
  \lambda_2(\Sigma)\vert \Sigma\vert^{2/n}\leq\frac{n}{(n+1)^2}I(\Omega)^{2+\frac{2}{n}},
\end{equation}
where $I(\Omega)$ is the isoperimetric ratio of $\Omega$, that is   
$$I(\Omega)=\frac{\vert \Sigma \vert}{\vert \Omega \vert^{n/(n+1)}},$$
where $\vert \Sigma\vert$ and $\vert \Omega \vert$ stand for the
Riemannian $n$-volume of $\Sigma$ and the Riemannian $(n+1)$-volume of
$\Omega$, respectively.  Moreover, equality holds in \eqref{chavel} if
and only if $\Sigma$ is embedded as a round sphere.  

\medskip
The main feature of the upper bound~\eqref{chavel}  is its low
sensitivity to small deformations, compared to that of the
curvature or the intersection index. This result of Reilly and
Chavel has been revisited by many authors \cite{AdoCR,GMO,WangXia},
but only for the first non-zero eigenvalue $\lambda_2$, and  using
barycentric type methods involving coordinate functions.

\medskip
Our aim in this paper is to establish inequalities of
Reilly-Chavel type for  higher order eigenvalues, that is to show that
the isoperimetric ratio $I(\Omega)$ allows a control  of the entire
spectrum of $\Sigma = \partial \Omega$, and in various ambient spaces.
Let us start with the particular but important case of compact
hypersurfaces in Euclidean space.

\begin{thm}\label{eucl}
  For any bounded domain $\Omega \subset \R^{n+1}$ with smooth
  boundary $\Sigma = \partial \Omega$, and all $k \ge 1$,  
  \begin{equation}\label{euc}
    \lambda_k(\Sigma)\vert \Sigma\vert^{2/n}\le \gamma_n I(\Omega)^{1+2/n}{k}^{2/n}
  \end{equation}
  with $\gamma_n = \frac{2^{10n+18+8/n}\  } {(n+1)}\
    \omega_{n+1}^{\frac 1{n+1}}$, where $\omega_{n+1}$ is the 
  volume of the unit ball in $\R^{n+1}$. 
\end{thm}
This result can also be understood as an estimate of the volume
prescribed by a Riemannian manifold once embedded as an hypersurface
in $\R^{n+1}$. That is, if $(\Sigma , g)$ is a Riemannian manifold
of dimension $n$ of volume one, then, for any isometric embedding
$\phi:\Sigma\rightarrow\mathbb{R}^{n+1}$, the domain $\Omega$ bounded
by the hypersurface $\phi(\Sigma)$ satisfies, for each $k\ge 2$,
\begin{gather}\label{preisom}
  \vert\Omega\vert^{\frac{n+2}{n+1}}\leq
  \gamma_n  \frac{k^{\frac 2{n}}}{\lambda_k(\Sigma)}.
\end{gather}
In particular, if the Riemannian metric $g$ is such that
$\lambda_k$ is large, then the prescribed volume $\vert\Omega\vert$
has to be small (see \cite[Theorem 1.4]{cde} for the existence of hypersurfaces with large $\lambda_k$).

\medskip
For more general ambient spaces, we have the following theorem which
is a particular case of a more general result (Theorem
\ref{thm-metric}) we will prove in section \ref{sectionHypersurface}
in which the curvature assumptions are replaced by hypotheses of
metric type.
\begin{thm} \label{main1}
  Let $(M,h)$ be a complete Riemannian manifold of dimension $n+1$
  with Ricci curvature bounded below by $-na^2$, $a\in\R$. For any
  bounded domain $\Omega \subset M$ with smooth  boundary
  $\Sigma=\partial\Omega$, and all $k \ge 1$, we have 
  \begin{equation}\label{generalRiem}
    \lambda_k(\Sigma) \le
    \alpha_n\frac{I(\Omega)}{I_0(\Omega)}a^2+\beta_n
    \left(\frac{I(\Omega)}{I_0(\Omega)}\right)^{1+2/n}
    \left(\frac{k}{\vert \Sigma \vert}\right)^{2/n},
  \end{equation}
  where 
  $$I_0(\Omega)=\inf \{I(U): U \text{is an open  set  in\ } \Omega \}$$
  and $\alpha_n$ and $\beta_n$ are two constants depending only on $n$
  (see \eqref{alpha,a=0} and \eqref{alpha,a=1} for explicit
  expressions of these constants).
\end{thm}

Observe that the power of $k$ appearing in the right hand side
of this estimate is optimal, according to Weyl's law.
  
\medskip
It is in general not easy to estimate the number $I_0(\Omega)$, which
represents the best constant in the isoperimetric inequality for
domains in $\Omega$. Recall that for any domain $\Omega$ in
$\R^{n+1}$, one has
$$I_0(\Omega)=I_0(\R^{n+1})= (n+1)\omega_{n+1}^{\frac 1{n+1}},$$
where $\omega_{n+1}$ denotes the volume of the unit ball in
$\R^{n+1}$.
In a Cartan-Hadamard manifold, it is  known that there exists a
universal positive constant $C_n$ such that  $I(\Omega)\ge C_n$ for any
bounded domain $\Omega$ (see \cite{croke2}).
More generally, if $(M,h)$ is any complete Riemannian manifold with
positive injectivity radius $\text{inj}(M)$, then  any domain $U$
contained in a geodesic ball of radius $r<\frac 12\text{inj}(M)$
satisfies $I(U)\ge C_n$ (see \cite{croke2} and \cite[Proposition
V.2.3]{cha3}). This leads to the following two corollaries.

\begin{cor}
 Let $(M,h)$ be a Cartan-Hadamard manifold of dimension $n+1$ with
 Ricci curvature bounded below by $-na^2$, $a\in \R$. For any bounded
 domain $\Omega \subset M$ with smooth boundary $\Sigma = \partial
 \Omega$, and all $k \ge 1$,
 \begin{equation}\label{cartan-hadamard}
   \lambda_k(\Sigma) \le 
   A_n I(\Omega)a^2+B_n I(\Omega)^{1+2/n}\left(\frac{k}{\vert \Sigma \vert}\right)^{2/n},
 \end{equation}
 where $A_n$ and $B_n$ are constants depending only on $n$.
\end{cor}

In view of~(\ref{chavel}), it would be interesting to know if the
first term on the right hand side of
inequality~(\ref{cartan-hadamard}) is necessary.
In
Example~\eqref{hyper} we will show that it is not always possible to
remove this term, at least if we allow the topology of $M$ to be
non-trivial.

\medskip
\begin{cor}
  Let $(M,h)$ be a complete Riemannian manifold of dimension $n+1$
  with Ricci curvature bounded below by $-na^2$, $a\in \R$, and positive
  injectivity radius. For any compact hypersurface $\Sigma$ bounding a
  domain $\Omega \subset M$ contained in a geodesic ball of
  radius $r<\frac 12\text{inj}(M)$, and for each $k \ge 1$, one has 
  \begin{equation}\label{injectivity}
    \lambda_k(\Sigma) \le A_n I(\Omega)a^2+B_n I(\Omega)^{1+2/n}\left(\frac{k}{\vert \Sigma \vert}\right)^{2/n},
  \end{equation}
  where $A_n$ and $B_n$ are two constants depending only on $n$.
  In particular,  for any bounded domain $\Omega$ in a hemisphere of
  the standard sphere $\bbS^{n+1}$ with smooth boundary $\Sigma
  = \partial \Omega$, and all $k \ge 1$, 
  $$
  \lambda_k(\Sigma) {\vert \Sigma \vert}^{2/n}\le B_n I(\Omega)^{1+2/n}{k}^{2/n}.
  $$
\end{cor}
 
\medskip
The assumption that the domain is contained in a geodesic ball of
radius $r<\frac 12\text{inj}(M)$ is necessary. Indeed,  in Example
\ref{torus} below, we will show that if $(M,h)$ is any compact
manifold, then there exists a sequence of domains for which inequality
\eqref{injectivity} fails, whatever the constants $A_n$ and $B_n$ are.

\medskip
Notice that it is impossible to obtain an inequality such
as~(\ref{injectivity}) for a class of domains $\Omega$ in a Riemannian
manifold $M$ without an assumption that guarantees that their
isoperimetric ratio 
$I(\Omega)$ is uniformly bounded from below. Indeed, since
$\lambda_k(\Sigma)\vert \Sigma\vert^{2/n}\sim c_n k^{2/n} 
\mbox{ as }k\rightarrow\infty$ (Weyl's asymptotic formula with
$c_n=4\pi^2 \omega_n^{-2/n}$),  the inequality (\ref{injectivity})
implies that $$I(\Omega)^{1+2/n}\geq c_n/B_n.$$

\medskip

Finally, let us mention that in our recent work ~\cite{ceg2}, we studied isoperimetric control of the Steklov
  spectrum for bounded domains in a complete Riemannian manifold. The
  methods we used in ~\cite{ceg2} are based on concentration properties
  which were initiated by Korevaar~\cite{kvr}, and further developed
  by Grigor'yan, Netrusov and Yau~\cite{gy,gyn}. Together with the
  results of the present paper, this leads to comparison
  results between the Steklov spectrum of a domain and the spectrum of
  its boundary hypersurface. See~\cite[Section 4]{ceg2} for details.

\section{Eigenvalue bounds : a general result}
\label{sectionHypersurface}
In this section, we give an upper bound for the eigenvalues of
the Laplacian in terms of quantities which depend only on the
Riemannian distance and measure.

Let $M$ be a Riemannian  manifold $M$ of dimension $n+1$.  
The Riemannian volumes of a geodesic ball $B(x,r)$  and of a geodesic
sphere $\partial B(x,r)$ of radius $r$ in $M$ are asymptotically
equivalent as $r\to 0$ to $\omega_{n+1} r^{n+1}$  and $ \rho_n r^{n}$,
respectively, where $\omega_{n+1}$ is the volume  of a unit ball and
$\rho_n=(n+1)\omega_{n+1}$ is the volume  of a unit sphere in the
$(n+1)$-dimensional Euclidean space.  
To each point $x$ in $M$ we associate the number $r(x)$ defined as the
largest positive number (possibly infinite) so that, fo all $r< r(x)$,
one has  
$$|B(x,r)| < 2 \omega_{n+1} r^{n+1}$$
and
$$|\partial B(x,r)| < 2 \rho_n r^{n}.$$
If $M$ has nonnegative Ricci curvature,  then, thanks to
the Bishop-Gromov inequality, $r(x)=+\infty$ for all $x\in M$.

Let $\Omega$ be a bounded regular domain in $M$ and denote by $\Sigma$
the boundary of $\Omega$. We define the  number $\xxx{r_{-}}(\Omega)$  as follows :
$$\xxx{r_{-}}(\Omega)=\inf_{x\in\Sigma} r(x).$$
We also introduce for all $r>0$, an integer $N_{M}(r)$
such that for any $x\in M$ and any $s<r$, the
geodesic ball  $ B(x,4s)$ can be covered by $N_{M}(r)$ balls of
radius $s$.

\medskip

\xxx{The main technical result of this paper is the following}
\begin{prop}\label{mainprop}
  Let $r_0$ be a positive number such that $r_0<\frac{1}{4}\xxx{r_{-}}(\Omega)$ and define
  $k_0$ to be the first integer satisfying
  $$k_0>\frac {1}{16\rho_n}\frac {I_0(\Omega)}{r_0^n}|\Omega|^{n/(n+1)}.$$
  For all $k\ge k_0$,
  $$\lambda_k(\Sigma)|\Sigma|^{\frac 2 n}\le
  256 \left(16\rho_n\right)^{\frac 2 n}
  N_{M}(r_0)^2 \left(\frac{I(\Omega)}{I_0(\Omega)}\right)^{1+\frac 2 n} k^{\frac 2 n}.$$
\end{prop}

\xxx{Proposition~\ref{mainprop} has the following consequence, from which the results announced in the introduction will follow.}
\begin{thm}\label{thm-metric}
  Let $M$ be a \xxx{complete} Riemannian manifold  of dimension $n+1$ and let
  $\Omega\subset M$ be a bounded  domain whose boundary  $\Sigma$ is a
  smooth hypersurface. For any $r_0<\frac{1}{4}\xxx{r_{-}}(\Omega)$ and any positive
  integer $k$, one has 
  \begin{equation}\label{generalMetric} 
    \lambda_k(\Sigma)\le
    256 N_{M}(r_0)^2
    \frac{I(\Omega)}{I_0(\Omega)}
    \left\{ \frac{1}{r_0 ^2} +
      \left(16\rho_n \frac{I(\Omega)}{I_0(\Omega)}
        \frac k{|\Sigma|}\right)^{\frac 2 n}
    \right\}.
 \end{equation}
\end{thm}

It is in general not easy to estimate the quantities $I_0(\Omega)$ and
$N_{M}(r_0)$ that appear in the right-hand side of this
inequality.
However, in many standard geometric situations  it is possible to
control these invariants in terms of the dimension and a lower bound of
the Ricci curvature.  This
will lead to the results stated in the introduction. For example, when
$M$ is the Euclidean space $\R^{n+1}$ one has for any
$\Omega\subset\R^{n+1}$,   $\xxx{r_{-}}(\Omega)=+\infty$,
$I_0(\Omega)=I_0(\R^{n+1})= (n+1) \omega_{n+1}^{\frac 1{n+1}}$ and
$N_{M}(r)\le {32}^{(n+1)} $ for all $r>0 $ (see Lemma
\ref{packing} below).\\ 

For the need of the proof, we endow $M$ with the
Borel measure $\mu$ with support in $\Sigma$ defined for each
Borelian $\mathcal{O}\subset M$ by 
$$\mu(\mathcal{O})=\int_{\mathcal{O}\cap\Sigma}dv_g,$$
In other words, the $\mu$-measure of $\mathcal O$ is the volume of the
part of the hypersurface $\Sigma$ lying inside 
$\mathcal O$. The geodesic distance of $M$ will be denoted  by $d$. 

One of the main tools in the proof is the following result which is an
adapted version of a result obtained by Maerten and the first author
in \cite{CM}:
 
\begin{lemma}\label{CMrevisited}
  Let $(X,d,\mu)$ be a complete, locally compact metric measure
  space, where $\mu$ is a finite measure. We assume that for all
  $r>0$, there exists an integer $N(r)$ such that each ball of radius
  $4r$ can be covered by $N(r)$ balls of radius $r$. If there
  exist an integer $K>0$ and a radius $r>0$ such that, 
  for each $x \in X$
  $$
  \mu(B(x,r)) \le \frac{\mu(X)}{4N^2(r)K},
  $$
  then, there exist $K$ $\mu$-measurable subsets $A_1,...,A_K$ of $X$
  such that, $\forall i\le K$, $\mu(A_i)\ge \frac{\mu(X)}{2N(r)K}$
  and, for
  $i\not =j$, $d(A_i,A_j) \ge 3r$.
\end{lemma}

The proof of this lemma consists of a slight modification of the
construction made in \cite[section 2]{CM}. For convenience, the
proof is included at the end of the paper.

\begin{proof}[Proof of Proposition \ref{mainprop}]
  The Rayleigh quotient of a function $f$ in the Sobolev space
  $H^1(\Sigma)$ is
  $$R(f)=\frac{\int_\Sigma|\nabla_{\Sigma}f|^2}{\int_\Sigma f^2}.$$
  The $k$-th eigenvalue $\lambda_k(\Sigma)$ is characterized as
  follows:
  \begin{gather*}
    \lambda_k(\Sigma)=\inf_E\sup_{0\neq f\in E}R(f)
  \end{gather*}
  where the infimum is over all $k$-dimensional subspaces of the
  Sobolev space $H^1(\Sigma)$ (see for instance~\cite{berard}). In
  particular, in order to obtain upper bounds on $\lambda_k$, we
  will construct $k$ test functions with disjoint supports and
  controlled Rayleigh quotient.
  
  \medskip
  Let us fix an integer $k\ge k_0$ and set
  \begin{gather}
    r_k
    =\left(\frac{I_0(\Omega)}{4^{n+2}\rho_n k}\right)^{1/n}|\Omega|^{\frac{1}{n+1}}
  \end{gather}
  so that $r_k^n\le
  \frac{1}{4^n}\frac{I_0(\Omega)|\Omega|^{\frac{n}{n+1}} }
  {16\rho_nk_0} <\left(\frac{r_0}4\right)^n$
  , that is 
  $r_k<\frac{r_0}4$.\\

\medskip
\noindent
\textbf{Step 1}\\
Let us first show that $\Sigma$ cannot be covered by $2k$ balls of
radius $4r_k$. \xxx{More precisely, 
let  $x_1,x_2,\dots,x_{2k}$ be $2k$ (arbitrary) points in $M$ and
define
\begin{gather*}
  {M_0}=M\setminus\cup_{j=1}^{2k}B(x_j,4r_k),\\
  {\Omega_0}=\Omega\setminus\cup_{j=1}^{2k}B(x_j,4r_k),\mbox{ and }
  {\Sigma_0}=\Sigma\setminus\cup_{j=1}^{2k}B(x_j,4r_k).
\end{gather*}}
Then, 
\begin{gather}\label{omega_0}
  |\Omega_0|>\frac 34 |\Omega|
\end{gather}
and 
\begin{equation}\label{sigma_0}
  |\Sigma_0|>\frac 12 I_0(\Omega) |\Omega|^{\frac{n}{n+1}} =\frac 12 \frac {I_0(\Omega)}{I(\Omega)}|\Sigma|,
\end{equation}

\medskip
\noindent
Indeed, since $4r_k<r_0<\xxx{r_{-}}(\Omega)$,
$$ \sum_{j=1}^{2k} |B(x_j,4r_k)|< 4k \omega_{n+1}(4r_k)^{n+1}$$ 
with 
$$(4r_k)^{n+1}=
\left(\frac {I_0(\Omega)}{16\rho_n k}\right)^{\frac {n+1}n}|\Omega|<
\frac 1{16 k}\left(\frac {I_0(\Omega)}{\rho_n }\right)^{\frac
  {n+1}n}|\Omega|\le
\frac 1{16 k \omega_{n+1}}|\Omega|$$
where the last inequality follows from the fact that
$$I_0(\Omega)\le I_0(\R^{n+1})=\frac{\rho_n}{\omega_{n+1}^{n/(n+1)}}.$$
Therefore,
$$ \sum_{j=1}^{2k} \xxx{|B(x_j,4r_k)|< \frac{1}{4 } |\Omega|}$$
and 
\begin{gather*}
  |\Omega_0|> |\Omega| -\frac{1}{4 } |\Omega|= \frac 34 |\Omega|.
\end{gather*}
Now, observe that the boundary of $\Omega_0$ consists of the union of
$\Sigma_0$ and parts of the boundaries of the balls
$B(x_j,4r_k)$. Therefore,
$$|\partial\Omega_0| \le
|\Sigma_0| +\sum_{j=1}^{2k} |\partial B(x_j,4r_k)|<
|\Sigma_0|+ 4k \rho_{n}(4r_k)^{n}=
|\Sigma_0|+ \frac 14 I_0(\Omega)  |\Omega|^{\frac n{n+1}}.$$
On the other hand, from the isoperimetric inequality satisfied by
domains in $\Omega$ and (\ref{omega_0}) we get
$$|\partial\Omega_0|
\ge I_0(\Omega) |\Omega_0|^{\frac n{n+1}}>
\left(\frac 34\right)^{\frac n{n+1}}I_0(\Omega) |\Omega|^{\frac n{n+1}}.$$
Hence,
$$|\Sigma_0| > 
\left[\left(\frac 34\right)^{\frac n{n+1}}-\frac 14\right] I_0(\Omega)
|\Omega|^{\frac n{n+1}}>\frac 12 I_0(\Omega)  |\Omega|^{\frac
  n{n+1}}.$$

\medskip
\noindent
\textbf{Step 2}\\
The result of the previous step makes it possible to define
inductively a family of $2k$ balls $B(x_{1}, r_k), \dots,
B(x_{2k},r_k)$ satisfying the following:
\xxx{\begin{gather*}
  \mu\left(B(x_{1},r_k)\right)
  =\sup_{x\in M}\mu(B(x,r_k)),\\
  \mu\left(B(x_{j+1},r_k)\right)
  =\sup\left\{\mu(B(x,r_k))\, :\, x\in M\setminus\cup_{i=1}^jB(x_i,4r_k)\right\}.
\end{gather*}}

\smallskip
It follows from this construction that 
\begin{itemize}
\item[a)] the balls $B(x_{1},2 r_k), \dots, B(x_{2k},2r_k)$ are mutually disjoint, 
\item[b)] $\mu(B(x_1,r_k))\ge \mu(B(x_2,r_k))\ge\cdots\ge \mu(B(x_{2k},r_k))$,
\item[c)] \xxx{$\forall x\in {M_0}=M\setminus\cup_{j=1}^{2k}B(x_j,4r_k)$, $\mu(B(x,r_k))\le \mu(B(x_{2k},r_k))$.}
\end{itemize}

\medskip
Two alternatives are to be considered separately, depending on how the
ball $B(x_{2k},r_k)$ is $\mu$-charged. This will be done in the two
following steps.
 
\medskip
\noindent\textbf{Step 3}\\
Assuming that
\begin{gather}\label{case1b}
 \mu\left(B(x_{2k},r_k)\right)\geq
 \frac {I_0(\Omega)  |\Omega|^{\frac n{n+1}}} {16k N_{M}(r_0)^2}=
 \frac {1} {16k N_{M}(r_0)^2}\frac {I_0(\Omega)} {I(\Omega)} |\Sigma|,
\end{gather}
we show that
$$\lambda_k(\Sigma)|\Sigma|^{\frac 2 n}\le
\frac{16 N_{M}(r_0)^2}{r_k^2 }\frac {I(\Omega)} {I_0(\Omega)}|\Sigma|^{\frac 2 n}.$$

\medskip
Indeed, for each $1\leq j\leq 2k$ we consider the function $f_j$
supported in $B(x_j,2r_k)$ and defined for all $x\in B(x_j,2r_k)$ by :
\begin{gather}
  f_j(x)=\min\left\{1,\ 2-\frac{1}{r_k}d(x_j,x)\right\}.
\end{gather}
Since $|\nabla f_j|^2\leq \frac{1}{r_k^2}$ in $B(x_j,2r_k)$, the
Rayleigh quotient of the restriction of $f_j$ to $\Sigma$, that we
still denote by $f_j$,  clearly satisfies

\begin{gather}\label{r(f)}
  R(f_j)\leq\frac{1}{r_k^2}\frac{\mu(B(x_j,2r_k))}{\mu(B(x_j,r_k))}
\end{gather}
with (from the definition of $x_1,\dots,x_{2k}$) 
$$\mu\left(B(x_{j},r_k)\right)\geq
\mu\left(B(x_{2k},r_k)\right)\geq
\frac {1} {16k N_{M}(r_0)^2}\frac {I_0(\Omega)}{I(\Omega)}|\Sigma|.$$

\medskip
On the other hand, the balls $B(x_j,2r_k)$,
$j=1,\dots,2k$, being mutually disjoint, there exist 
$k$ of them, $B(x_{j_1},2r_k), \dots, B(x_{j_k},2r_k)$ satisfying
\begin{gather*}
  \mu(B(x_{j_m},2r_k))\leq|\Sigma|/k\quad \mbox{ for }m=1,\dots,k.
\end{gather*}
Replacing into \eqref{r(f)} we get, $\forall m=1,\dots,k$,
\begin{gather*}
  R(f_{j_m})< \frac{16 N_{M}(r_0)^2}{r_k^2 }  \frac {I(\Omega)} {I_0(\Omega)}
\end{gather*}
so that
$$\lambda_k(\Sigma)|\Sigma|^{\frac 2 n}\leq
\max_{1\le m\le k}R(f_{j_m})|\Sigma|^{\frac 2 n}\leq
\frac{16 N_{M}(r_0)^2}{r_k^2 }  \frac {I(\Omega)} {I_0(\Omega)}|\Sigma|^{\frac 2 n}.$$

\medskip
\noindent\textbf{Step 4}\\
Assuming now that
\begin{gather}\label{case1a}
 \mu\left(B(x_{2k},r_k)\right)< \frac {1} {16k N_{M}(r_0)^2}\frac {I_0(\Omega)} {I(\Omega)} |\Sigma|,
\end{gather}
we show that
$$\lambda_k(\Sigma)|\Sigma|^{\frac 2 n}\le
\frac{8  N_{M}(r_0)}{r_k^2}
\frac{I(\Omega)}{I_0(\Omega)}|\Sigma|^{\frac 2 n}.$$

\medskip
Indeed, from the construction of the balls $B(x_{j},r_k)$ (see step
2), \xxx{one has, $\forall x\in
{M_0}=M\setminus\cup_{j=1}^{2k}B(x_j,4r_k)$},
$$\mu(B(x,r_k))\le \mu(B(x_{2k},r_k))<
\frac {1} {16k N_{M}(r_0)^2}\frac {I_0(\Omega)} {I(\Omega)} |\Sigma|$$
with $\mu(\Sigma_0) = |\Sigma_0|>\frac 12 \frac{I_0(\Omega)}{I(\Omega)} |\Sigma|$ 
(see \eqref{sigma_0}). \xxx{Hence, $\forall x\in {M_0}$, we have}

\begin{eqnarray} \label{control}
4 N_{M}(r_0)^2\mu(B(x,r_k))< \frac{\mu(\Sigma_0)} {2k}.
\end{eqnarray}

\medskip
This enables us to apply Lemma \ref{CMrevisited}  with $K=2k$ and
$r=r_k$ to the metric
  measure space  $M$ endowed with the Riemannian distance $d$ and the restriction $\mu_0$ of the measure
$\mu$ to $\Sigma_0$, namely for a Borelian $\mathcal O \subset M$, we have 
$\mu_0(\mathcal O)= \mu(\mathcal O\cap \Sigma_0)$. In particular, 

\begin{eqnarray} \label{mesure}
\mu_0(M) =\mu(\Sigma_0) = |\Sigma_0|>\frac 12 \frac{I_0(\Omega)}{I(\Omega)} |\Sigma|.
\end{eqnarray}

\smallskip
The relation (\ref{control}) becomes
\begin{eqnarray}
4 N_{M}(r_0)^2\mu_0(B(x,r_k))< \frac{\mu_0(M)} {2k}.
\end{eqnarray}

\medskip
Thus, we deduce the existence of $2k$ measurable sets
$A_1,\dots,A_{2k}$ \xxx{in $M_0$ satisfying both}~:
$$\mu(A_i)\ge\frac{\mu_0(\Sigma_0)}{4 k N_{M}(r_0)}\quad\mbox{ for all }i$$
and
$d(A_i,A_j)\ge 3r_k$ if $i\neq j$.
Denote by 
$$A_i^{r_k}=\{x \in M \; ;\; d(x,A_i)<r_k\}$$
the $r_k$-neighborhood of $A_i$.  
A priori, we have no control over $\mu_0(A_i^{r_k})$, but since
$d(A_i,A_j) \geq 3r$ for $i \neq j$, the $A_i^{r_k}$ are mutually
disjoint and there exist $k$ sets amongst them, say $ A_1^{r_k},
\dots, A_k^{r_k}$, which satisfy  
\[
\mu_0(A_i^{r_k}) \le\frac{|\Sigma|}{k} \quad \mbox{ for }i=1,\dots,k. 
\]       
As in \cite{CM}, we construct for each $i\le k$,  a test function
$\varphi_i$ with support in $A_i^{r_k}$  and which is defined for all
$x\in A_i^{r_k}$ by   
$$\varphi_i(x)=1-\frac{d(x,A_i)}{r_k}.$$ 
Observing that $|\nabla \varphi_i(x)|\le \frac 1 {r_k}$ almost
everywhere in $A_i^{r_k}$, a straightforward calculation shows that the
Rayleigh quotient of  the restriction of  $\varphi_i$ to $\Sigma$,
that we still denote by $\varphi_i$,  satisfies  
$$
R(\varphi_i) \le \frac{1}{r_k^2}\frac{\mu_0(A_i^r)}{\mu_0(A_i)} <
\frac{1}{r_k^2}\frac{|\Sigma|}{\frac{\mu_0(M)}
  {4N_{M}(r_0)}}
$$

and, because of (\ref{mesure}), we have 
$$
R(\varphi_i) \le \frac{8  N_{M}(r_0)}{r_k^2}
\frac{I(\Omega)}{I_0(\Omega)}.$$ 
Thus, 
$$\lambda_k(\Sigma)|\Sigma|^{\frac 2 n}\le \max_{1\le i\le
  k}R(\varphi_{i})|\Sigma|^{\frac 2 n}\le
\frac{8  N_{M}(r_0)}{r_k^2} \frac{I(\Omega)}{I_0(\Omega)}|\Sigma|^{\frac
  2 n}.$$ 

\medskip
\noindent\textbf{Step 5}\\
We are now ready to conclude the proof.

\medskip
From the two previous steps, we
see that in all cases, one has 
$$\lambda_k(\Sigma)|\Sigma|^{\frac 2 n}\le
\frac{16 N_{M}(r_0)^2}{r_k^2 }  \frac {I(\Omega)} {I_0(\Omega)}|\Sigma|^{\frac 2 n}$$
with
$\frac{|\Sigma|^{\frac 2 n}}{r_k^2}=
\left(\frac{4^{ n+2}\rho_n k}{I_0(\Omega)}\right)^{\frac 2 n}
\frac{|\Sigma|^{\frac 2 n}}{\Omega^{\frac 2{n+1}}}=
\left(4^{ n+2}\rho_n k\right)^{\frac 2 n}
\left(\frac{I(\Omega)}{I_0(\Omega)}\right)^{\frac 2 n} $. Thus, 
$$\lambda_k(\Sigma)|\Sigma|^{\frac 2 n}\le
256 \left(16\rho_n\right)^{\frac 2 n} N_{M}(r_0)^2
\left(\frac{I(\Omega)}{I_0(\Omega)}\right)^{1+\frac 2 n} k^{\frac 2 n}.$$

\end{proof}

\begin{proof}[Proof of Theorem \ref{thm-metric}]
  Let $k$ be a positive integer. If $k<k_0$, then
  $\lambda_k(\Sigma)\le\lambda_{k_0}(\Sigma).$
  Together with Proposition \ref{mainprop}, this yields for all $k\ge 1$,
  $$\lambda_k(\Sigma)|\Sigma|^{\frac 2 n}\le
  256 \left(16\rho_n\right)^{\frac 2 n}
  N_{M}(r_0)^2
  \left(\frac{I(\Omega)}{I_0(\Omega)}\right)^{1+\frac 2 n}\max
  \left\{{k_0}^{\frac 2 n} ,  k^{\frac 2 n}\right\}.$$
  We clearly have
  $\max\left\{{k_0}^{\frac 2 n},  k^{\frac 2 n}\right\}\le
  (k_0-1)^{\frac 2 n}+k^{\frac 2 n}$, 
  with $k_0-1\le \frac {1}{16\rho_n}\frac {I_0(\Omega)}{r_0^n}|\Omega|^{n/(n+1)}$. Consequently
  $$\lambda_k(\Sigma)|\Sigma|^{\frac 2 n}\le
  256 N_{M}(r_0)^2
  \left\{
    \frac{I(\Omega)^{1+\frac 2 n}}{I_0(\Omega)}
    \frac{|\Omega|^{2/(n+1)}}{r_0^2} +
    \left(16\rho_n\right)^{\frac 2 n}
    \left(\frac{I(\Omega)}{I_0(\Omega)}\right)^{1+\frac 2 n}
    k^{\frac 2 n}
  \right\}.$$ 
  Replacing $|\Omega|^{2/(n+1)}$  by $\frac{|\Sigma|^{\frac 2 n}}{I(\Omega)^{\frac 2 n}}$, we get
  $$ \lambda_k(\Sigma)\le 256 N_{M}(r_0)^2
  \frac{I(\Omega)}{I_0(\Omega)} \left\{ \frac{1}{r_0^2} +
    \left(16\rho_n \frac{I(\Omega)}{I_0(\Omega)}
      \frac k{|\Sigma|}\right)^{\frac 2 n}\right\}.$$
\end{proof}

\section{Proof of Theorems \ref{eucl} and \ref{main1} and 
comments}\label{proof-theorem}
The proofs of these theorems rely  on Bishop-Gromov comparison
results and the following packing lemma (see \cite{zhu} Lemma 3.6)

\begin{lemma}\label{packing} 
  Let $(X,d,\nu)$ be a locally compact metric measure space and let
  $r$, $R$ and $V$ be positive numbers with $r<R$ and such that,
  $\forall x \in X$,
  $$
  \frac{\nu(B(x,2R))}{\nu(B(x,\frac r 4))} \le V.
  $$
  Then each ball of radius $R$ in $X$ can be covered by $\lfloor
  V\rfloor$ balls of radius $r$.
  In particular, when the ambient space is the standard  $\R^{n+1}$,
  then any ball of radius $R$ can be covered by  $\left(8\frac R
    r\right)^{n+1}$ balls of radius $r$.
\end{lemma}

\begin{proof}[Proof of  Theorem \ref{eucl}]
  Let $\Omega$ be a bounded domain in $\R^{n+1}$. With the notations of
  the last section, one clearly has $\xxx{r_{-}}(\Omega)=\infty$,
  $$I_0(\Omega)=I_0(\R^{n+1})= (n+1) \omega_{n+1}^{\frac 1{n+1}}$$
  and, from Lemma \ref{packing}, $\forall r>0$,
  $N_M(r)\le{32}^{(n+1)}$.  We then apply Theorem \ref{thm-metric}
  to get, after letting $r_0$ 
  go to infinity, $\forall k\ge 1$, 
  $$\lambda_k(\Sigma) {\vert \Sigma \vert}^{2/n}\le \gamma_n I(\Omega)^{1+2/n}{k}^{2/n}$$
  with $\gamma_n =  \frac{2^{10n+18+8/n}\  } {(n+1)}\  \omega_{n+1}^{\frac 1{n+1}}$. 
\end{proof}

\begin{proof}[Proof of  Theorem \ref{main1}]
Let $(M,h)$ be a complete Riemannian manifold of dimension $n+1$ whose
Ricci curvature tensor satisfies 
$$Ric\ge -n a^2 h$$
for some constant $a$. Let $\Omega$ be a bounded domain in $M$ with
regular boundary $\Sigma=\partial \Omega$. To prove  Theorem
\ref{main1}, we treat separately the three following  cases:\\

\medskip
\noindent - \textit{Case $a=0$},  that is $M$ has nonnegative Ricci
curvature. From Bishop-Gromov comparison results (see
\cite[p.156]{sak}) we deduce that for all $x\in M$ and all $r>0$,
$|B(x,r)|\le \omega_{n+1}r^{n+1}$, $|\partial B(x,r)|\le
\rho_{n}r^{n}$ and
$$\frac{\vert B(x,8r)\vert}{\vert B(x,r/4)\vert}
\le {32}^{(n+1)}.$$
Thus,
$\xxx{r_{-}}(\Omega)=\infty$  and, applying Lemma \ref{packing},
$N_M(r)\le {32}^{(n+1)} $ for all $r>0$. Replacing in
\eqref{generalMetric} and letting $r_0$ go to infinity we get, for all $k\ge
1$,
$$\lambda_k(\Sigma)|\Sigma|^{\frac 2 n}\le 
{2}^{10(n+1)+8}\left(16\rho_n\right)^{\frac 2 n}  
\left(\frac{I(\Omega)}{I_0(\Omega)}\right)^{1+\frac 2 n} k^{\frac 2 n}.$$

\medskip
\noindent
- \textit{Case $a=1$}.  The volumes of a ball and  of a
sphere of radius $r$ in the hyperbolic space ${\mathbb H}^{n+1}$ of
curvature $-1$ and dimension $n+1$ are given by
$$V_{-1}(n,r)=\rho_n\int_0^r\left(\sinh s\right)^n ds\quad\mbox{ and }\quad
S_{-1}(n,r)=\rho_n \left(\sinh r\right)^n.$$
We define the constant $r(n)$ to be the largest $r>0$  such that $
\left(\sinh r\right)^n \leq 2 r^n$ and set  
$$V(n)=\sup_{0<r<r(n)}
\frac{V_{-1}(n,8r)}{V_{-1}(n,r/4)}.$$
Again, the Bishop-Gromov comparison theorem gives, for all $x\in M$
and all positive $r<r(n)$,
\begin{gather*}
  |B(x,r)|\le V_{-1}(n,r)<2\omega_{n+1}r^{n+1},\quad
  |\partial B(x,r)|\le S_{-1}(n,r)<2\rho_{n}r^{n},\\
    \frac{\vert B(x,8r)\vert}{\vert
      B(x,r/4)\vert} \le \frac {V_{-1}(n,8r)} {V_{-1}(n,r/4)}\le
    V(n).
\end{gather*}

\medskip
Thus,  $\xxx{r_{-}}(\Omega)\ge r(n)>0$  and, applying Lemma
\ref{packing}, $N_M(r)\le V(n) $ for all $r<r(n)$. Applying
Theorem \ref{thm-metric} we get 
$$
\lambda_k(\Sigma)\le 256 V^2(n) \frac{I(\Omega)}{I_0(\Omega)} \left\{
  \frac{1}{r ^2(n)} + \left(16\rho_n \frac{I(\Omega)}{I_0(\Omega)}
    \frac k{|\Sigma|}\right)^{\frac 2 n}\right\}$$ 
that is
\begin{equation}\label{a=1}
\lambda_k(\Sigma) \le \alpha_n\frac{I(\Omega)}{I_0(\Omega)}+\beta_n
\left(\frac{I(\Omega)}{I_0(\Omega)}\right)^{1+ 2/
  n}\left(\frac{k}{\vert \Sigma \vert}\right)^{2/n}, 
\end{equation}
with $\alpha_n = 256 \frac{V^2(n)}{r ^2(n)}$ and
$\beta_n=\left(16\rho_n\right)^{\frac 2 n} \alpha_n$.   

\medskip
\noindent - \textit{Case $a\neq 0$}. The metric $\tilde h =a^2 h$ is
such that $Ric_{\tilde h}\ge -n \tilde h$.  The metric $\tilde g=a^2
g$ induced on $\Sigma $ by $\tilde h$ is so that
$\lambda_k(\Sigma,\tilde g)=\frac 1 {a^2}\lambda_k(\Sigma)$ and
$\vert (\Sigma, \tilde g) \vert = a^n \vert \Sigma \vert$ while the isoperimetric ratio is
invariant under scaling. Thus, applying the inequality \eqref{a=1} to
$\Omega$ considered as a domain in $(M, \tilde h)$ we get 
$$\frac 1 {a^2}\lambda_k(\Sigma)\le
\alpha_n\frac{I(\Omega)}{I_0(\Omega)}+\beta_n \left(\frac{I(\Omega)}
{I_0(\Omega)}\right)^{1+2/n}\left(\frac{k}{a^n\vert \Sigma \vert}\right)^{2/n},$$
which gives, after simplification,  the desired inequality. 

In conclusion, inequality \eqref{generalRiem} is proved with
\begin{gather}\label{alpha,a=0}
  \alpha_n=0\quad\mbox{ and }  \quad \beta_n = {2}^{10(n+1)+8}\left(16\rho_n\right)^{\frac 2 n}
\end{gather}
if $a= 0$, and
\begin{gather}\label{alpha,a=1}
  \alpha_n= 256 \frac{V^2(n)}{r ^2(n)}\quad\mbox{and}
  \qquad \beta_n = \left(16\rho_n\right)^{\frac 2 n} \alpha_n
\end{gather}
if $a\ne 0$. 
\end{proof}

\begin{exa}\label{hyper}
  The aim of the following construction is to show that the
  boundedness of the sectional curvature does not suffice to get an
  estimate such as~(\ref{euc}). Indeed, 
  We will construct a sequence of manifolds $(M_i,h_i)$ whose sectional curvature is
  between $-1$ and 0, each containing a domain $\Omega_i$ with
  boundary $\Sigma_i$ such that
  $I(\Omega_i)=1$, $\vert \Sigma_i \vert $ tends 
  to infinity with $i$, and  $\lambda_3(\Sigma_i)$ is
  bounded below by a positive constant.
\end{exa}

\medskip
According to \cite{brooks}, there exists a sequence $N_i$ of compact
hyperbolic manifolds of dimension $n\ge 2$ whose volume tends to
infinity with $i$ while $\lambda_2(N_i)$  does not converge to
zero. We can assume that $\vert N_i\vert >i$ and
$\lambda_2(N_i)>C$ for some positive constant $C$. 
For each $i$, set $M_i=N_i\times \R$ and $\Omega_i=N_i\times (-L_i,
L_i) \subset M_i$ with $L_i=\left(2 \vert
  N_i\vert\right)^{\frac{1}n}$. We endow $M_i$ with the product metric,
so that the sectional curvature of $M_i$ is between $-1$ and zero. The
boundary $\Sigma_i$ of  the domain $\Omega_i$ consists of two disjoint
copies of $N_i$. Therefore,
$\lambda_1(\Sigma_i)=\lambda_2(\Sigma_i)=0$ and
$\lambda_3(\Sigma_i)=\lambda_2(N_i)>C$. On the other hand, we have
$\vert \Sigma_i \vert = 2\vert N_i\vert >2i$, $\vert \Omega_i \vert =
2 L_i \vert N_i\vert= \left(2\vert N_i\vert \right)^{\frac{n+1}n}$
and $I(\Omega_i)=1$.

\begin{exa}\label{torus}
Let $(M,h)$ be any compact Riemannian manifold of dimension $n+1\ge 3$. Then there exists a sequence $\Omega_i$ of domains in $M$ with smooth boundaries $\Sigma_i$, and a positive constant $C$ such that $\lambda_2(\Sigma_i)\vert \Sigma_i \vert^{\frac2n}\ge C$  while $\vert \Sigma_i \vert$ and $I(\Omega_i)$ go to zero as $i$ tends to infinity.  
\end{exa}
In particular, there exist no constants $A_n$ and $B_n$ such that the sequence  $\Omega_i$ satisfies an inequality like  \eqref{injectivity}.

Let us first assume that $(M,h)$ is flat in a geodesic ball $B(x_0,r)$ centered at some $x_0\in M$. This ball is isometric to a Euclidean ball of radius $r$. Let $S_i$ be a  sequence of Euclidean spheres of radius $r/i$ that we embed isometrically as hypersurfaces $\Sigma_i$ into $B(x_0,r)\subset M$.  For each $i$,   the hypersurface $\Sigma_i$ bounds a domain $\Omega_i$ which contains the complement of $B(x_0,r)$. This sequence of domains $\Omega_i$ satisfies :
$$\vert \Omega_i \vert =\vert M \vert -\omega_{n+1}\left(\frac r i\right)^{n+1},$$
$$\vert \Sigma_i \vert =\rho_n\left(\frac r i\right)^{n},$$
that is, $\vert \Sigma_i \vert$ and $I(\Omega_i)$ go  to zero a $i$ tends to infinity. On the other hand, 
$$\lambda_2(\Sigma_i)\vert \Sigma_i \vert^{\frac2n} =\lambda_2(S_i)\vert S_i \vert^{\frac2n} = n\rho_n^{\frac2n}.$$

Now, for a general compact Riemannian manifold  $(M,h)$, it is possible to deform the metric $h$ into  a metric $h'$ which is quasi-isometric to $h$ with quasi-isometry ratio close to 1, and so that  $(M,h')$ is flat in a small geodesic ball $B(x_0,r)$. The sequence  $\Omega_i$ of domains constructed above with respect to $h'$ would  be such that, for the metric $h$,  $\vert \Sigma_i \vert$ and $I(\Omega_i)$ go  to zero a $i$ tends to infinity while $\lambda_2(\Sigma_i)\vert \Sigma_i \vert^{\frac2n}$ is bounded below by a positive constant.

\section{Proof of Lemma \ref{CMrevisited}}\label{proof-lemma-CM}

Let $(X,d,\mu)$ be a complete, locally compact metric measure
space, where $\mu$ is a finite measure. 
We assume that for all $r>0$, there exists an integer $N(r)$ such that
each ball of radius $4r$ can be covered by $N(r)$ balls of radius
$r$. Let us first prove the following :
\begin{lemma}\label{step1}
  Let $\beta$ be a positive number satisfying $\beta \le
  \frac{\mu(X)}{2}$, and let $r>0$ be such that, for all $x \in X$,
  $$
  \mu\left(B(x,r) \right)\le \frac{\beta}{2N(r)}.
  $$
  Then there exist two open subsets $A$ and $D$ of $X$ with $A \subset D$,
  such that $\mu(A)\ge \beta$, $\mu(D)\le 2N(r)\beta$ and
  $d(A,D^c) \ge 3r$.
\end{lemma}

\begin{proof}[Proof of Lemma~\ref{step1}]
  For each positive integer  $m$ we denote by $\mathcal{U}_{m}(r)$ the
  set of unions of $m$ balls of radius $r$, that is,  
  $$
  \mathcal{U}_{m}(r):=\left\{ \bigcup^{m}_{j=1}
    B\left(x^{j},r\right)\; ; \;    x^1,\dots, x^m\in X \right\},  
  $$
  and consider the evaluation $\Psi_{m}$ of the measure $\mu$
  on $\mathcal{U}_{m}(r)$, that is
  $${\Psi_{m}:X^{m}= \underbrace{X\times X \times \cdots \times X }_{m \  {\rm times}} \longrightarrow   \R }$$ 
  with
  $$
  \Psi_{m}(x^1,\dots, x^m)= \mu\left(\bigcup^{m}_{j=1} B\left(x^{j},r\right) \right).
  $$
  Since $(X,d)$ is a complete  locally compact metric space and
  $\mu(X)<+\infty$, the function $\Psi_{m}$ achieves its maximum
  $\xi(m)$ at  some point $\textbf{a}_{m}=(a_m^1,\dots, a_m^m) \in
  X^{m}$ (not necessary unique), that is, 
  $$ 
  \mu  \left(\bigcup^{m}_{j=1} B\left(x^{j},r\right)\right) \le
  \mu  \left(\bigcup^{m}_{j=1} B\left(a_m^{j},r\right)\right)
  $$
  for any $(x^1,\dots, x^m)\in X^m$. 
  
  Now, from the assumptions of the Lemma one clearly has $\xi(1)\le
  \frac{\beta}{2N(r)}\le \beta$. On the other hand, for $m$ large
  enough, we necessarily have $\xi(m)  \ge\frac{ 3\beta}{2}$ (indeed, it
  suffices to consider a ball $B(z,R)$ satisfying $\mu ( B(z,R))\ge
  \frac 34\mu(X)$ and notice that it can be covered with a finite number
  of balls of radius $r$). In conclusion, there exists an integer $k\ge
  2$ such that $\xi(k)\ge \beta$ and  $\xi(k-1)\le \beta$.
  
  We  set $A:=\underset{1\le j \le k}{\bigcup} B\left(a^{j}_{k},r\right) $ 
  and $D:=\underset{1\le j \le k}{\bigcup} B\left(a^{j}_{k},4r\right)
  $. From their definitions, these sets satisfy $\mu(A)=\xi (k)\ge
  \beta$,  and $d(A,D^c) \ge 3r$. We still need to check that $\mu(D)\le
  2N(r)\beta$. Indeed, according to our hypotheses, each ball
  $B\left(a^{j}_{k},4r\right)$ can be covered by $N(r)$ balls of
  radius $r$. Hence, $D$ can be covered by $kN(r)$ balls of radius
  $r$, namely
  $D\subset\underset{1\le j \le kN(r)}{\bigcup}B_{j}$, 
  where the $B_{j}$ are balls of radius $r$. From
    $kN(r) \le 2(k-1)N(r)$, it follows that this union of
  balls can be written as 
  $\underset{1\le j \le kN(r)}{\bigcup}B_{j}= \underset{1\le j \le
    2N(r)}{\bigcup}W_{j}$ where 
  each $W_{j}\in \mathcal{U}_{k-1}(r)$. It follows that
  \begin{eqnarray*}
    \mu(D)\le 
    \mu\left(\bigcup^{2N(r)}_{j=1}W_{j}\right) 
    \le \sum^{2N(4r)}_{j=1}\mu(W_{j}) 
    \le 2 N(r) \xi (k-1)\le 2 N(r) \beta.
  \end{eqnarray*}
\end{proof}

\begin{proof}[Proof of Lemma \ref{CMrevisited}]
  Let $K\ge 2$ be an integer and $r>0$ a positive number such that,
  $\forall x \in X$
  $$
  \mu(B(x,r)) \le \frac{\mu(X)}{4N^2(r)K}.
  $$
  Our aim is to construct  $K$ $\mu$-measurable subsets $A_1,...,A_K$
  of $X$ such that, $\forall i\le K$, $\mu(A_i)\ge
  \frac{\mu(X)}{2N(r)K}$ and, for  
  $i\not =j$, $d(A_i,A_j) \ge 3r$.
  
  For simplicity, we set $\alpha  = \frac{\mu(X)}{2N(r)K}$. We shall
  construct, using a finite induction,  $K$ pairs
  $\left(A_{1},D_{1}\right), \dots, \left(A_{K},D_{K}\right)$ of 
  sets such that, $\forall j\le K$,
\begin{enumerate}

\item[(0)]
$A_j\subset D_j$

\item 
$A_j \subset (\cup_{i=1}^{j-1}D_i)^c $

 \item 
	$\mu(A_{j})  \ge  \alpha$
	
	\item	
	$\mu(D_{j})  \le  2N(r)\alpha  = \frac{\mu(X)}{K}$
	
	\item	
$ d(A_{j},(\cup_{i=1}^{j}D_i)^c )  \ge  3r .$
 
\end{enumerate}
Indeed, the family $A_1,...,A_K$ will then satisfy the desired properties since $\mu(A_{j})  \ge  \alpha$ and, 
if $k<j$, $d(A_k,A_j) \ge d(A_{k},(\cup_{i=1}^{k}D_i)^c )  \ge  3r$ (notice that   
$A_j\subset (\cup_{i=1}^{j-1}D_i)^c \subset (\cup_{i=1}^{k}D_i)^c$ since $k<j$). 

 To initiate the iteration it suffices to apply Lemma \ref{step1} with
 $\beta=\alpha$. Therefore, there exist two open sets $A_{1}$ and
 $D_{1}$ satisfying $A_{1}\subset D_{1}$ and 
 $$\left\{
   \begin{array}{rll}
     \mu(A_{1})&  \ge & \alpha\\
     \mu(D_{1}) & \le & 2N(r)\alpha = \frac{\mu(X)}{K}\\
     d(A_{1},   D_{1}^{c}) & \ge & 3r .
   \end{array}\right. $$

Now, assume that we have already constructed, for a certain $j<K$,  $j$ couples $\left(A_{1},D_{1}\right), \dots, \left(A_{j},D_{j}\right)$ satisfying the induction hypothesis.  
We endow $X$ with the measure $\mu_{j+1}$ defined by $\mu_{j+1}(U)=\mu(U \cap (\cup_{i=1}^{j}D_i)^c)$.

From the induction hypothesis, one has 
$$\mu_{j+1}(X)=\mu((\cup_{i=1}^{j}D_i)^c) \ge \mu(X) -\sum_{i=1}^{j}\mu(D_i)\ge \mu(X)(1-\frac{j}{K})\ge \frac{\mu(X)}{K}.$$
Therefore, for all 
$x \in X$, one has
$$
\mu_{j+1}(B(x,r)) \le \mu(B(x,r)) \le \frac{\mu(X)}{4N^2(r)K} \le \frac{\mu_{j+1}(X)}{4N^2(r)}.
$$
This allows us to apply Lemma  \ref{step1} to the metric measure
space $(X, d, \mu_{j+1})$  with
$\beta= \alpha=\frac{\mu(X)}{2N(r)K}\le \frac{\mu_{j+1}(X)}{2N(r)}$. Thus,  
there exist two open sets  $A$ and $D$ satisfying $A \subset D$,
$\mu_{j+1}(A)\ge \alpha$,
$\mu_{j+1}(D) \le 2N(r)\alpha=\frac{\mu(X)}{K}$ and $d(A,D^c) \ge
3r$. We define the couple $\left(A_{j+1},D_{j+1}\right)$ by 
$$A_{j+1}=A\cap(\cup_{i=1}^{j}D_i)^c \qquad , \qquad D_{j+1}
=D \cap(\cup_{i=1}^{j}D_i)^c.$$ 
It remains to check that the family $\left\{\left(A_{1},D_{1}\right),
  \dots, \left(A_{j+1},D_{j+1}\right)\right\}$ satisfies the induction
hypothesis. Indeed, the three first properties of this hypothesis are
immediate consequences of the construction. To see that $
d(A_{j+1},(\cup_{i=1}^{j+1}D_i)^c )  \ge  3r $ we only need to observe
that $D= D_{j+1}\cup \left(D \cap \cup_{i=1}^{j}D_i \right)\subset
\cup_{i=1}^{j+1}D_i $ which implies  $
d(A_{j+1},(\cup_{i=1}^{j+1}D_i)^c )  \ge d(A,D^c) \ge  3r $.

\end{proof}

\bibliographystyle{plain}
\bibliography{biblioCEG}

\end{document}